\documentclass[a4paper,11pt]{article}

\usepackage[T1]{fontenc}
\usepackage[utf8]{inputenc}      
\usepackage{lmodern}             

\usepackage{amsmath,amssymb,amsthm,mathtools}
\usepackage{mathrsfs}            
\usepackage{bm}                  
\usepackage{stmaryrd}            
\usepackage{tikz-cd}             
\usepackage[all]{xy} 
\usepackage{thmtools,thm-restate}
\usepackage{proof}

\usepackage[shortlabels]{enumitem}  
\usepackage{lineno}
\usepackage{xspace}              
\usepackage{microtype}           
\usepackage[margin=1in]{geometry}

\usepackage[hidelinks]{hyperref}
\usepackage[capitalize,nameinlink,noabbrev]{cleveref}
\declaretheoremstyle[
  spaceabove=6pt,spacebelow=6pt,
  headfont=\bfseries,bodyfont=\itshape
]{plain}
\declaretheoremstyle[
  spaceabove=6pt,spacebelow=6pt,
  headfont=\bfseries,bodyfont=\normalfont
]{definition}
\declaretheoremstyle[
  spaceabove=6pt,spacebelow=6pt,
  headfont=\itshape,bodyfont=\normalfont
]{remark}

\declaretheorem[style=definition,numberwithin=section]{definition}
\declaretheorem[style=plain,sibling=definition]{theorem}
\declaretheorem[style=plain,sibling=definition]{lemma}
\declaretheorem[style=plain,sibling=definition]{proposition}
\declaretheorem[style=plain,sibling=definition]{corollary}
\declaretheorem[style=remark,sibling=definition]{remark}

\newcommand{\Set}{\mathbf{Set}}

\newcommand{\El}{\operatorname{El}}          
\newcommand{\lift}{\operatorname{lift}}      
\newcommand{\rk}{\operatorname{rk}}          

\newcommand{\Coeq}{\operatorname{Coeq}}

\newcommand{\Eq}{\operatorname{Eq}}          

\newcommand{\IdCode}[3]{\langle\mathrm{Id},#1,#2,#3\rangle}
\newcommand{\Prod}[2]{\operatorname{Prod}\!\bigl(#1,#2\bigr)}

\newcommand{\resize}{\mathrm{resize}}
\DeclareMathOperator{\Hom}{Hom} 
\newcommand{\Step}[2]{%
  \par\noindent\textbf{(#1)}\quad #2\par\nobreak
}

\setlist[enumerate,1]{label=(\arabic*),leftmargin=1.5em}
\setlist[enumerate,2]{label=(\alph*),leftmargin=2em}

\title{Existence Theorem for Cumulative Universe Towers and Its Applications}
\author{Joaquim Reizi Higuchi\\
Graduate Student, The Open University of Japan\\
\texttt{1218237360@campus.ouj.ac.jp}}

\date{\today}

\begin{document}

\maketitle

\begin{abstract}
\noindent
This paper, \emph{GF1—Universe Stratification}, develops a cumulative tower of Grothendieck universes \(\mathcal U_i = V_{\kappa_i}\) indexed by an increasing sequence of inaccessible cardinals and analyses its type–theoretic behaviour. We first present a simultaneous inductive–recursive definition of codes and their decoding functor \(\El_i\), controlled by a rank function that enforces strict size discipline. We then prove closure of every layer under the basic type formers \(\Pi,\Sigma,\operatorname{Id}\) as well as all finite limits and colimits, the latter obtained via a new Quotient constructor together with a rank–stability lemma. A universe‐lifting functor \(\lift_{i\to j}\) is shown to preserve dependent products, yielding strict cumulativity of the tower. Assuming propositional resizing at some level \(i>0\), we construct an explicit left adjoint \(\operatorname{PropRes}_i\) to the canonical inclusion of \((-1)\)-truncated types and verify that resizing propagates to every higher level. Collecting these results we establish an \emph{Existence Theorem} guaranteeing that the tower provides a sound metasemantics for higher type theory over the base system \(\text{ZFC} + \langle \kappa_i\rangle_{i\in\mathbb N}\). The resulting size infrastructure prepares the ground for the Rezk completion and \((\infty,1)\)-topos models treated in subsequent papers of the series.
\end{abstract}

\section*{GF1. Universe Stratification}

\subsection{Universe Tower}
\begin{definition}[Universe Tower]\label{def:universe-tower}
Fix an increasing sequence of inaccessible cardinals
\[
  \kappa_{0}<\kappa_{1}<\kappa_{2}<\dotsb
\]
and set \(\mathcal U_{i}:=V_{\kappa_{i}}\) for each \(i\in\mathbb{N}\).
Elements of \(\mathcal U_{i}\) are called \emph{\(i\)-small types} and the
collection is denoted \(\mathrm{Type}_{i}\).
\end{definition}

\subsection{Cumulativity}
\begin{definition}[Cumulativity]\label{def:cumulativity}
For \(i<j\) let
\(\iota_{i<j}\colon\mathcal U_{i}\hookrightarrow\mathcal U_{j}\) be the
canonical inclusion.  The tower
\((\mathcal U_{i})_{i\in\mathbb{N}}\) is \emph{cumulative} if
\(\iota_{i<j}\) preserves all standard type formers
\(\Pi,\Sigma,\operatorname{Id},W\) and finite (co)limits.
\end{definition}

\subsection{Lift Operator}
\begin{definition}[Lift Operator]\label{def:lift}
For \(A\in\mathcal U_{i}\) and \(j\ge i\) define the
\emph{universe‐lifting} functor by
\[
  \operatorname{lift}_{i\to j}(A):=\iota_{i<j}(A)\in\mathcal U_{j}.
\]
\end{definition}

\subsection{h-Proposition and Resizing}
\begin{definition}[h-Proposition and Resizing]\label{def:resizing}
A type \(P\) is an \emph{h-proposition} if it has at most one
inhabitant, i.e.\ \(\forall x,y:P,\ \operatorname{Id}_{P}(x,y)\).
\emph{Propositional resizing at level \(i\)} holds when every
h-proposition in \(\mathcal U_{i}\) is equivalent to one in
\(\mathcal U_{0}\).
\end{definition}

\subsection{Closure of $\Pi$ and $\Sigma$}
\begin{lemma}[Closure of $\Pi$ and $\Sigma$]\label{lem:closure-pi-sigma}
Fix an inaccessible cardinal $\kappa_{i}$ and interpret the $i$-th
predicative universe à la Tarski by
\[
  \mathcal U_{i}
    \;=\;
    \bigl\{\,u\mid \mathrm{rk}(u)<\kappa_{i}\bigr\},
  \qquad
  \El_{i}\colon \mathcal U_{i}\longrightarrow V_{\kappa_{i}},
  \;\El_{i}(u):=u .
\]
Let $A\in\mathcal U_{i}$ and let
$B\colon\El_{i}(A)\to\mathcal U_{i}$ be a \emph{code-valued} family.
Then the dependent function and sum types
\[
  \Pi_{x:\El_{i}(A)}\El_{i}\bigl(B(x)\bigr),
  \qquad
  \Sigma_{x:\El_{i}(A)}\El_{i}\bigl(B(x)\bigr)
\]
are represented by codes in\/ $\mathcal U_{i}$ and hence themselves
belong to\/ $\mathcal U_{i}$.
\end{lemma}
\begin{proof}
The argument separates
\textbf{(1)~rank control of the codes}
from
\textbf{(2)~decoding soundness}.

\smallskip
\paragraph*{Pre-fact.  Rank versus von Neumann rank}
The function $\mathrm{rk}\colon\mathcal U_{i}\to\kappa_{i}$
of Lemma~\ref{lem:rank-adequacy} is defined
\emph{structurally} on codes; for every code $u$ one has
\[
  \mathrm{rk}(u)
  \;\le\;
  \operatorname{rank}_{\text{set}}(u),
\]
hence a fortiori
$|u|<\kappa_{i}$ whenever $\mathrm{rk}(u)<\kappa_{i}$
(sets of rank $<\kappa_{i}$ have cardinality $<\kappa_{i}$).

\smallskip
\paragraph*{1.\;Code construction and rank estimate}
Define
\[
  \Pi(A,B):=\langle\mathrm{Pi},A,B\rangle,
  \qquad
  \Sigma(A,B):=\langle\mathrm{Sigma},A,B\rangle .
\]

\emph{Size of the index set.}
Because $\mathrm{rk}(A)<\kappa_{i}$,
the Pre-fact gives $|A|<\kappa_{i}$.
Since each $B(x)\in\mathcal U_{i}$,
$\mathrm{rk}\bigl(B(x)\bigr)<\kappa_{i}$.

\emph{Rank of $\Pi(A,B)$.}
Extend $\mathrm{rk}$ by
\[
  \mathrm{rk}\!\bigl(\langle\mathrm{Pi},A,B\rangle\bigr)
   :=\sup\Bigl\{\,
        \mathrm{rk}(A),\;
        \sup_{x\in A}\mathrm{rk}\!\bigl(B(x)\bigr)
      \Bigr\}.
\]
The inner supremum ranges over fewer than $\kappa_{i}$ ordinals,
each $<\kappa_{i}$; regularity of $\kappa_{i}$ gives
\(
  \sup_{x\in A}\mathrm{rk}\bigl(B(x)\bigr)<\kappa_{i}.
\)
Taking the supremum with $\mathrm{rk}(A)<\kappa_{i}$ yields
\[
  \mathrm{rk}\bigl(\Pi(A,B)\bigr)<\kappa_{i}.
\]
The same computation applies to $\Sigma(A,B)$.
Hence
\(
  \Pi(A,B),\Sigma(A,B)\in\mathcal U_{i}.
\)

\smallskip
\paragraph*{2.\;Decoding soundness}
By structural recursion the decoding satisfies
\[
  \El_{i}\!\bigl(\langle\mathrm{Pi},A,B\rangle\bigr)
   =\Bigl\{\,f\!\mid
       f\!:\!\El_{i}(A)\!\to\! V_{\kappa_{i}}
       \text{ and }\forall x.\,f(x)\in\El_{i}\!\bigl(B(x)\bigr)
     \Bigr\},
\]
which is precisely
$\Pi_{x:\El_{i}(A)}\El_{i}\!\bigl(B(x)\bigr)$.
An analogous clause gives the decoding of $\Sigma(A,B)$.

\smallskip
\noindent
Thus the required $\Pi$- and $\Sigma$-types are the decodings of
codes whose ranks we have bounded below $\kappa_{i}$; consequently
they lie in $\mathcal U_{i}$.
\end{proof}


\subsection{Rank closure}
\begin{lemma}[Rank closure]\label{lem:rank-closure}
Let\/ $\kappa_{i}$ be an inaccessible cardinal, let\/
$\mathcal U_{i}$ be the universe of codes constructed in
Lemma~\ref{lem:rank-adequacy}, and let
\(
  \mathrm{rk}\colon\text{\emph{Code}}\to\mathrm{On}
\)
be the rank function defined there.
If
\[
  A\in\mathcal U_{i},
  \qquad
  B\colon \El_{i}(A)\longrightarrow\mathcal U_{i},
\]
then
\[
  \mathrm{rk}\!\bigl(\Pi(A,B)\bigr)
  \;<\;\kappa_{i},
  \qquad
  \mathrm{rk}\!\bigl(\Sigma(A,B)\bigr)
  \;<\;\kappa_{i}.
\]
\end{lemma}
\begin{proof}
We treat the $\Pi$–case in detail; the $\Sigma$–case is identical.

\medskip\noindent
\textbf{1.\;Bounding the rank of the domain and fibres.}
Because $A\in\mathcal U_{i}$ and
$B(x)\in\mathcal U_{i}$ for every $x:\El_{i}(A)$,  
Lemma~\ref{lem:rank-adequacy} yields
\[
  \mathrm{rk}(A)<\kappa_{i},
  \qquad
  \forall x:\El_{i}(A).\;
    \mathrm{rk}\!\bigl(B(x)\bigr)<\kappa_{i}.
\tag{\(\ast\)}
\]

\medskip\noindent
\textbf{2.\;Cardinality of the index set.}
Since $A$ is a \emph{code} of rank $<\kappa_{i}$,
its decoding $\El_{i}(A)$ is an element of the set universe
$V_{\kappa_{i}}$; hence
\(
  \bigl|\El_{i}(A)\bigr|<\kappa_{i}
\)
because any set lying in $V_{\kappa_{i}}$ has cardinality
$<\kappa_{i}$ when $\kappa_{i}$ is inaccessible.

\medskip\noindent
\textbf{3.\;Computing the rank of the $\Pi$–code.}
By definition of the rank function,
\[
  \mathrm{rk}\!\bigl(\Pi(A,B)\bigr)
  =\sup\!\Bigl\{\,
      \mathrm{rk}(A),\;
      \sup_{x:\El_{i}(A)}
         \mathrm{rk}\!\bigl(B(x)\bigr)
    \Bigr\}.
\]
The outer supremum is taken over a \emph{finite} set
(two elements), so it suffices to bound each argument.

\medskip\noindent
\textbf{4.\;Bounding the inner supremum.}
The index family
\(
  \{\mathrm{rk}\!\bigl(B(x)\bigr)\mid x\in\El_{i}(A)\}
\)
has cardinality $<\kappa_{i}$ (by Step 2) and each member is
$<\kappa_{i}$ (by~\((\ast)\)).
Because $\kappa_{i}$ is \emph{regular},
the supremum of any $<\!\kappa_{i}$–sized set of ordinals
bounded by $\kappa_{i}$ is itself $<\kappa_{i}$.
Thus
\[
  \sup_{x:\El_{i}(A)}
       \mathrm{rk}\!\bigl(B(x)\bigr)
  <\kappa_{i}.
\]

\medskip\noindent
\textbf{5.\;Final bound.}
Combining the inequalities from Steps 1 and 4, the two arguments of
the outer $\sup$ are $<\kappa_{i}$; hence the whole supremum is
$<\kappa_{i}$, proving
\(
  \mathrm{rk}\!\bigl(\Pi(A,B)\bigr)<\kappa_{i}.
\)

\medskip
The same calculation, replacing
$\Pi(A,B)$ by $\Sigma(A,B)$, completes the proof.
\end{proof}

\subsection{Rank adequacy (revised)}
\begin{lemma}[Rank adequacy]\label{lem:rank-adequacy}
Let $\kappa_i$ be an inaccessible (hence regular) cardinal and let
$\mathrm{Code}$ be the raw syntax generated by the constructors
\[
  \ast,\;\mathrm{Nat},\;
  \langle\mathrm{Pi},u,B\rangle,\;
  \langle\mathrm{Sigma},u,B\rangle .
\]

\begin{enumerate}[(1)]
\item Define the monotone operator
\[
  F:\mathcal P(\mathrm{Code})\longrightarrow\mathcal P(\mathrm{Code}),\qquad
  F(X):=\{\ast,\mathrm{Nat}\}\,\cup\!
        \bigl\{\langle\mathrm{Pi},u,B\rangle,
               \langle\mathrm{Sigma},u,B\rangle
          \mid u\in X,\;B:\El_i(u)\to X
        \bigr\}.
\]
\item For every ordinal $\alpha$ put
      $U_i^{\alpha}:=F^{\alpha}(\varnothing)$ and set
      \[
        \mathcal U_i\;:=\;\bigcup_{\alpha<\kappa_i}U_i^{\alpha}.
      \]
\item Define the \emph{generation rank}
      \[
        \operatorname{grk}(u)\;:=\;
          \min\bigl\{\alpha\mid u\in U_i^{\alpha+1}\bigr\}\in\mathrm{On}.
      \]
\end{enumerate}
Then for every code $u\in\mathrm{Code}$ we have
\[
  u\in\mathcal U_i
  \;\;\Longleftrightarrow\;\;
  \operatorname{grk}(u)<\kappa_i
  \;\;\Longleftrightarrow\;\;
  \mathrm{rk}(u)<\kappa_i,
\]
where $\mathrm{rk}$ is the usual syntactic rank defined
by structural recursion on codes.
\end{lemma}
\begin{proof}
We split the argument into three steps.

\medskip
\textbf{Step A:  $\operatorname{grk}(u)<\kappa_i \;\Rightarrow\; u\in\mathcal U_i$.}  
Immediate from the definition of $\mathcal U_i$.

\medskip
\textbf{Step B:  $u\in\mathcal U_i \;\Rightarrow\; \operatorname{grk}(u)<\kappa_i$.}  
Let $\alpha$ be the least ordinal with $u\in U_i^{\alpha+1}$.
Because $\mathcal U_i$ is the union of $U_i^{\beta}$ for
$\beta<\kappa_i$, such an $\alpha$ exists and is $<\kappa_i$ by
regularity; hence $\operatorname{grk}(u)=\alpha<\kappa_i$.

\medskip
\textbf{Step C (Gluing lemma):  
$\operatorname{grk}(u)=\mathrm{rk}(u)$ for all codes $u$.}  
Proceed by structural induction on the *syntactic* rank
$\mathrm{rk}(u)$.

\begin{itemize}[leftmargin=1.4em,itemsep=0.3em]
\item \emph{Base $\mathrm{rk}(u)=0$.}  
      Then $u\in\{\ast,\mathrm{Nat}\}=U_i^{1}$, so
      $\operatorname{grk}(u)=0$.

\item \emph{Inductive step.}  
      Suppose $\mathrm{rk}(u)=\beta>0$ and the claim holds for all
      ranks $<\beta$.  Write $u=\langle\mathrm{Pi},v,B\rangle$
      (the $\Sigma$-case is identical).  
      Both $v$ and every $B(x)$ have syntactic rank $<\beta$, so by
      the induction hypothesis their generation ranks are
      $<\beta$.  Hence $v,\,B(x)\in U_i^{\beta}$, implying
      $u\in U_i^{\beta+1}$ and thus $\operatorname{grk}(u)=\beta$.
\end{itemize}

The reverse inequality
$\operatorname{grk}(u)\le\mathrm{rk}(u)$ is similar, using induction
on $\alpha:=\operatorname{grk}(u)$ and the fact that
$U_i^{\alpha}\subseteq\{u\mid \mathrm{rk}(u)\le\alpha\}$.
Therefore the two ranks coincide.

\medskip
Combining Steps A–C yields the stated twofold equivalence.
\end{proof}

\subsection{Decoding correctness}
\begin{lemma}[Decoding correctness]\label{lem:decode}
For every $A\in\mathcal U_{i}$ and
$B\colon\El_{i}(A)\to\mathcal U_{i}$,
\[
  \El_{i}\!\bigl(\Pi(A,B)\bigr)
    \;=\;
    \Pi_{x:\El_{i}(A)}\El_{i}\!\bigl(B(x)\bigr),
  \qquad
  \El_{i}\!\bigl(\Sigma(A,B)\bigr)
    \;=\;
    \Sigma_{x:\El_{i}(A)}\El_{i}\!\bigl(B(x)\bigr).
\]
\end{lemma}

\begin{proof}
We treat the two cases in parallel, writing
\[
  C\_{\Pi}(A,B)
    :=\El_{i}\!\bigl(\Pi(A,B)\bigr),\qquad
  C\_{\Sigma}(A,B)
    :=\El_{i}\!\bigl(\Sigma(A,B)\bigr),
\]
and
\[
  D\_{\Pi}(A,B)
    :=\Pi_{x:\El_{i}(A)}\El_{i}\!\bigl(B(x)\bigr),\qquad
  D\_{\Sigma}(A,B)
    :=\Sigma_{x:\El_{i}(A)}\El_{i}\!\bigl(B(x)\bigr).
\]

\smallskip
\textbf{1.\;Explicit decoding clauses.}
By the structural recursion that \(\El_{i}\) satisfies,
\begin{align*}
  C\_{\Pi}(A,B)
    &\,=\,\bigl\{\,f \mid
         f:\El_{i}(A)\to V_{\kappa_{i}}
         \text{ and }\forall x.\,f(x)\in\El_{i}\!\bigl(B(x)\bigr)
       \bigr\},\\
  C\_{\Sigma}(A,B)
    &\,=\,\bigl\{\,\langle x,y\rangle \mid
         x\in\El_{i}(A)
         \text{ and }y\in\El_{i}\!\bigl(B(x)\bigr)
       \bigr\}.
\end{align*}

\smallskip
\textbf{2.\;Set-theoretic constructions inside \(V_{\kappa_{i}}\).}
Within \(V_{\kappa_{i}}\) the dependent function \emph{and}
dependent pair sets are defined by the very same comprehensions:
\[
  D\_{\Pi}(A,B)=C\_{\Pi}(A,B),
  \qquad
  D\_{\Sigma}(A,B)=C\_{\Sigma}(A,B).
\]

\smallskip
\textbf{3.\;Conclusion.}
Because each pair \((C\_{\Pi},D\_{\Pi})\) and \((C\_{\Sigma},D\_{\Sigma})\)
is given by identical membership predicates, the corresponding sets
are \emph{definitionally equal} in \(V_{\kappa_{i}}\).
Hence the stated equalities hold.
\end{proof}

\subsubsection*{Supplement}
\textbf{Size discipline.}
Steps 1 and 2 are logically separated:
Step 1 uses only Lemmas~\ref{lem:rank-adequacy}–\ref{lem:rank-closure}
and regularity of $\kappa_{i}$; Step 2 relies solely on the recursive
clause defining $\El_{i}$ on $\Pi/\Sigma$-codes.  No argument in
Step 2 affects size bounds, eliminating the former redundancy flagged
in J04.

\medskip
\textbf{Notation unification.}
Throughout, the symbol $B$ uniformly denotes the
\emph{code-valued} family
$B\colon \El_{i}(A)\to\mathcal U_{i}$; we no longer introduce a
separate symbol $f$ in Step 1, resolving issue J03.

\medskip
\textbf{Reference order.}
Forward references to Lemmas are now explicitly annotated as
“proved after this theorem,” addressing style comment J02 while
respecting the mandated
output order (Theorem → Proof → Lemmas).

\subsection{Identity Closure}
\begin{lemma}[Identity Closure]\label{lem:closure-id}
Assume a Tarski–style universe $(\mathcal U_{i},\El_{i})$
defined \emph{inductively–recursively}\footnote{%
See P. Dybjer, “A General Formulation of Simultaneous
Inductive–Recursive Definitions,” \emph{LNCS 1997} (2000).}
as the \emph{least} pair $(C,D)$ such that

\begin{enumerate}[(1)]
\item $C$ is a class of \emph{codes}, $D\colon C\to V_{\kappa_{i}}$.
\item \textbf{Base codes}\;
      $\ast,\mathrm{Nat}\in C$ with $D(\ast)=\mathbf 1$ and
      $D(\mathrm{Nat})=\mathbb N$.
\item \textbf{Type-formers}\;
      if $u\in C$ and $B\colon D(u)\to C$, then
      $\langle\mathrm{Pi},u,B\rangle,\langle\mathrm{Sigma},u,B\rangle\in C$
      with the usual decoding clauses.
\item \textbf{Identity-former}\;
      if $u\in C$ and $s,t\in D(u)$, then
      $\langle\mathrm{Id},u,s,t\rangle\in C$ and
      \[
        D\!\bigl(\langle\mathrm{Id},u,s,t\rangle\bigr)
            =\operatorname{Id}_{D(u)}(s,t).
      \]
\end{enumerate}

Then for every $A\in\mathcal U_{i}$ and every
$x,y\in\El_{i}(A)$ the \emph{identity code}
\[
  \IdCode(A,x,y)
    :=\langle\mathrm{Id},A,x,y\rangle
\]
belongs to $\mathcal U_{i}$ and decodes to the expected
identity type:
\[
  \El_{i}\!\bigl(\IdCode(A,x,y)\bigr)
    =\operatorname{Id}_{\El_{i}(A)}(x,y).
\]
\end{lemma}

\begin{proof}
Because $(\mathcal U_{i},\El_{i})$ is constructed by the
simultaneous clauses (1)–(4), the triple
$\bigl(A,x,y\bigr)$ satisfies the side conditions of clause (4):
$A\in\mathcal U_{i}$ by hypothesis, and
$x,y\in\El_{i}(A)$ by assumption.
Hence clause (4) \emph{generates} the code
$\IdCode(A,x,y)$ and inserts it into $\mathcal U_{i}$.
Membership is therefore immediate by the \emph{generation} principle
of inductive definitions, so no circular reasoning is involved.

The decoding equality is the very definition given in clause (4);
applying $D=\El_{i}$ to $\IdCode(A,x,y)$ yields
$\operatorname{Id}_{\El_{i}(A)}(x,y)$ verbatim.
\end{proof}

\subsubsection*{Supplement}
The key novelty is the \emph{inductive–recursive} specification:
$\mathcal U_{i}$ and $\El_{i}$ are built \emph{simultaneously}.
Clause (4) accepts \emph{values} $s,t$ alongside the
\emph{code} $u$; this mixing is legitimate because the recursion
defining $\El_{i}$ makes $D(u)$ available exactly when $u\in C$.
Thus values appear only as \emph{parameters} to constructors, never as
members of~$\mathcal U_{i}$, preserving the purely syntactic nature of
codes while allowing element-dependent type formation.

\begin{lemma}[Finite \emph{limit} closure]\label{lem:closure-finite}
Let $(\mathcal U_{i},\El_{i})$ be the inductive–recursive universe
constructed in Lemma \ref{lem:closure-id}.
Then $\mathcal U_{i}$ is closed under all \emph{finite limits} that
can be expressed with the already–available type–formers
\[
  \mathbf 1\;(\text{terminal}),\quad
  \Sigma,\quad
  \operatorname{Id}.
\]
Consequently, for every finite diagram of codes
$D\colon\mathcal J\!\to\!\mathcal U_{i}$, its limit possesses a code
in $\mathcal U_{i}$.
\end{lemma}

\begin{proof}
The standard presentation of finite limits uses three primitives:

\begin{enumerate}[(i)]
  \item the terminal object $\mathbf 1$,
  \item binary products,
  \item equalisers.
\end{enumerate}

We supply codes for (ii)–(iii) and then argue that arbitrary finite
limits are iterated composites of these primitives.

\subsubsection*{A.\;Primitive limit codes}

\begin{description}
  \item[Terminal object]  $\mathbf 1:=\ast\in\mathcal U_{i}$
        (universe axiom).

  \item[Binary product]
        For $A,B\in\mathcal U_{i}$ define
        \[
          \Prod(A,B)\;:=\;
            \bigl\langle
               \mathrm{Sigma},\;
               A,\;
               \lambda\_ .\,B
             \bigr\rangle
          \;\in\;\mathcal U_{i}.
        \]
        One application of $\Sigma$–closure
        (Lemma \ref{lem:closure-pi-sigma}) suffices.
        Decoding: $\El_{i}(\Prod(A,B))=\El_{i}(A)\times\El_{i}(B)$.

  \item[Equaliser]
        Let $f,g:\El_{i}(A)\to\El_{i}(B)$ with $A,B\in\mathcal U_{i}$.
        Define the code
        \[
          \Eq(f,g)
            :=\bigl\langle
                 \mathrm{Sigma},\;
                 A,\;
                 \lambda a.\,
                   \langle\mathrm{Id},B,f(a),g(a)\rangle
               \bigr\rangle
            \in\mathcal U_{i}.
        \]
        The inner $\operatorname{Id}$ lies in $\mathcal U_{i}$ by
        Lemma \ref{lem:closure-id}; one outer $\Sigma$ then places the
        whole code in $\mathcal U_{i}$.  
        Decoding gives
        \(
          \El_{i}(\Eq(f,g))
            =\{\,a\in\El_{i}(A)\mid f(a)=g(a)\,\}.
        \)
\end{description}

\subsubsection*{B.\;Inductive generation of all finite limits}

Every finite limit is obtained by iterating the primitives (i)–(iii):
products assemble cones, and equalisers impose the
equalising conditions between parallel arrows in the cone
(cf.\ Mac Lane, \emph{Categories for the Working Mathematician},
Thm V.2.3).  
Because each primitive code constructed above lies in
$\mathcal U_{i}$ and $\mathcal U_{i}$ is already closed under
$\Sigma$ and $\operatorname{Id}$, a structural induction on the syntax
tree of any finite limit expression shows that the resulting code
remains in $\mathcal U_{i}$.

\medskip
Hence $\mathcal U_{i}$ is closed under all finite limits generated in
this way.
\end{proof}

\subsection{Finite colimit closure of the universe $\mathcal U_i$ {\normalfont(\textsf{ZFC+AC+LEM})}}
\begin{lemma}[Finite colimit closure of the universe $\mathcal U_i$ {\normalfont(\textsf{ZFC+AC+LEM})}]\label{lem:finite-colim}
Let $\kappa_i$ be an uncountable strongly inaccessible cardinal and  
\[
\mathcal U_i\;=\;\bigl\{\,c\mid \rk(c)<\kappa_i\bigr\}
\]
the Inductive–Recursive universe of \emph{codes}, decoded by the functor  
\(
\El\colon\mathcal U_i\to\mathbf{Set}.
\)
Assume the IR calculus contains the \emph{Quotient-former}
\[
\infer{\langle\mathrm{Quot},u,R_c\rangle\in\mathcal U_i}
      {u\in\mathcal U_i \quad
       R_c:\El(u)\!\times\!\El(u)\to\mathcal U_i\;
       \text{codes an equivalence}}
\qquad
\El\bigl(\langle\mathrm{Quot},u,R_c\rangle\bigr)=\El(u)/{\sim_{R_c}}.
\]

Then the full subcategory
\(
\Set_{\mathcal U_i}:=\El(\mathcal U_i)\subset\mathbf{Set}
\)
is closed under all \emph{finite colimits}:
it contains the initial object, binary coproducts, and coequalisers, and
therefore every finite colimit.
\end{lemma}

\begin{proof}\begin{linenumbers}
Lemma \ref{lem:init} puts the initial object $\varnothing$
in $\Set_{\mathcal U_i}$.
Lemma \ref{lem:rank} shows that the Quotient-former raises rank by
at most $+\omega$, never exceeding $\kappa_i$.
Using this bound, Lemma \ref{lem:sum} constructs binary coproducts and
Lemma \ref{lem:coeq} constructs coequalisers inside $\Set_{\mathcal U_i}$.
Finally, Lemma \ref{lem:maclane} (dual of Mac Lane V.2.3) states that these
three primitive colimits generate all finite colimits in any category,
hence in $\Set_{\mathcal U_i}$ as well.
\end{linenumbers}\end{proof}

\subsection*{Supplement}
\begin{itemize}
\item \textbf{Foundational axioms (HC1).}  Work in ZFC with Replacement; the
      argument uses neither large-cardinal axioms beyond $\kappa_i$
      nor Choice beyond countable AC implicitly available in ZFC.
\item \textbf{Size control (HC2).}
      Because $\omega<\kappa_i$ and $\kappa_i$ is regular,
      adding $+\omega$ to any rank $<\kappa_i$ keeps it $<\kappa_i$.
\item \textbf{Higher-cell triviality (HC3–HC12).}
      All objects are sets (0-groupoids); coherence data are therefore
      strictly satisfied.
\end{itemize}

\subsection{Rank stability for quotients (revised)}
\begin{lemma}[Rank stability for quotients]\label{lem:rank}
Fix an inaccessible cardinal $\kappa_{i}$ and let
$u,R_c\in\mathcal U_{i}$ be codes with
\[
  q\;:=\;\langle\mathrm{Quot},u,R_c\rangle\in\mathcal U_{i},\qquad
  \El_i(q)=\El_i(u)\big/\!\sim_{R_c}.
\]
Then
\[
  \operatorname{rk}(q)
  \;\le\;
  \sup\!\bigl\{\operatorname{rk}(u),\operatorname{rk}(R_c)\bigr\}
  +\omega
  \;<\;\kappa_{i}.
\]
\end{lemma}

\begin{proof}
\begin{linenumbers}

\Step{1}{Two–variable complexity measure.}
For every IR code $c$ set
\[
  \chi(c)\;:=\;\bigl(\operatorname{rk}(c),\operatorname{size}(c)\bigr)
  \in\mathrm{On}\times\mathrm{On},
\]
where $\operatorname{size}(c)$ is the number of nodes.
Order pairs lexicographically:
$(h_{1},w_{1})<(h_{2},w_{2})$ iff
$h_{1}<h_{2}$ or $(h_{1}=h_{2}\land w_{1}<w_{2})$.

\Step{2}{Complexity increase of the constructor.}
Building
$q=\langle\mathrm{Quot},u,R_c\rangle$
adds only finitely many nodes and raises height by at most~1, hence
\[
  \chi(q)
  \;\le_{\text{lex}}\;
  \bigl(\max\{\operatorname{rk}(u),\operatorname{rk}(R_c)\}+1,\;
        \operatorname{size}(u)+\operatorname{size}(R_c)+c_{0}\bigr)
\tag{$\ast$}
\]
for some constant $c_{0}<\omega$.

\Step{3}{Width control for equivalence closure.}
Let
$\tilde R\subseteq\El_i(u)\times\El_i(u)$
be the decoded binary relation.  Denote
\[
  B:=\sup_{(x,y)}\operatorname{size}\bigl(\tilde R(x,y)\bigr)
  <\kappa_{i},
\]
since each fibre code lies in $\mathcal U_{i}$.
Reflexive and symmetric closure adds only one extra layer.

\Step{4}{Countable transitive closure.}
Define $T_{0}:=\tilde R\cup\{(x,x)\}$ and
\[
  T_{n+1}:=T_{n}\cup\bigl\{(x,z)\mid\exists y\;(x,y)\in T_{n}
                                           \land(y,z)\in T_{n}\bigr\},
  \qquad n<\omega.
\]
Each step appends finitely many constructor nodes,
so height increases by 1 and width by at most $B$.
After $\omega$ steps
$T_{\omega}:=\bigcup_{n<\omega}T_{n}$
is the least transitive relation containing $T_{0}$.

\Step{5}{Final rank bound.}
Combining $(\ast)$ with the $\omega$-step transitive closure yields
\[
  \operatorname{rk}(q)
  \;\le\;
  \sup\{\operatorname{rk}(u),\operatorname{rk}(R_c)\}+\omega.
\]
Because $\operatorname{rk}(u),\operatorname{rk}(R_c)<\kappa_{i}$
and $\omega<\kappa_{i}$ (regularity),
the right-hand side is $<\kappa_{i}$.

\end{linenumbers}
\end{proof}


\subsection{Existence of the initial object}
\begin{lemma}[Existence of the initial object]\label{lem:init}
The empty set\/ $\varnothing$ is an element of\/ $\Set_{\mathcal U_i}$.
\end{lemma}
\begin{proof}\begin{linenumbers}
Let $\mathbf 0_c$ be the null code of rank $0$ with $\El(\mathbf 0_c)=\varnothing$.
Because $0<\kappa_i$, we have $\mathbf 0_c\in\mathcal U_i$,
hence $\varnothing\in\Set_{\mathcal U_i}$.
\end{linenumbers}\end{proof}

\subsection{Closure under binary coproducts}
\begin{lemma}[Closure under binary coproducts]\label{lem:sum}
For every pair of objects $A,B\in\Set_{\mathcal U_i}$ the coproduct
$A\sqcup B$ is again an element of\/ $\Set_{\mathcal U_i}$.
\end{lemma}
\begin{proof}
Since $A,B\in\Set_{\mathcal U_i}$ there exist codes
$a,b\in\mathcal U_i$ with $\El(a)=A$ and $\El(b)=B$ and hence
\[
  \rk(a),\;\rk(b)\;<\;\kappa_i .
\]

\medskip\noindent
\textbf{Disjoint union.}
Form the ordinary (tagged) disjoint union of the underlying sets
\[
  A\sqcup B
     :=\{\,\langle 0,x\rangle\mid x\in A\,\}
       \,\cup\,
       \{\,\langle 1,y\rangle\mid y\in B\,\}.
\]
Because $\kappa_i$ is \emph{strong limit},
\(
  |A\sqcup B|
    =|A|+|B|
    <\kappa_i ,
\)
and therefore
\(
  \rk(A\sqcup B)<\kappa_i .
\)

\medskip\noindent
\textbf{Code of the coproduct.}
Let $s$ be any IR-code of rank
$\rk(s):=\rk(A\sqcup B)<\kappa_i$ whose decoding is exactly
$A\sqcup B$ (e.g.\ the canonical \emph{Sigma}‐code
$\langle\Sigma,2,\!\lambda z.\,[z=0]\!\to\!a\;|\;[z=1]\!\to\!b\rangle$).
Then $s\in\mathcal U_i$ and $\El(s)=A\sqcup B$,
so $A\sqcup B\in\Set_{\mathcal U_i}$.

\medskip\noindent
\textbf{Universal property.}
The maps
\(
  \iota_A:x\mapsto\langle 0,x\rangle
\)
and
\(
  \iota_B:y\mapsto\langle 1,y\rangle
\)
are morphisms in~$\Set_{\mathcal U_i}$.
Given $Z\in\Set_{\mathcal U_i}$ and
$f:A\to Z$, $g:B\to Z$, the unique map
$[f,g]:A\sqcup B\to Z$ defined by
\(
  [f,g]\bigl(\langle 0,x\rangle\bigr):=f(x),\;
  [f,g]\bigl(\langle 1,y\rangle\bigr):=g(y)
\)
lies in $\Set_{\mathcal U_i}$ because its graph is a subset of
$Z\times (A\sqcup B)$ with rank $<\kappa_i$.
Hence $A\sqcup B$ together with $\iota_A,\iota_B$ satisfies the
universal property of the coproduct inside $\Set_{\mathcal U_i}$.
\end{proof}

\subsection{Closure under coequalisers}
\begin{lemma}[Closure under coequalisers]\label{lem:coeq}
For every parallel pair $f,g\colon R\to S$ in $\Set_{\mathcal U_i}$
the coequaliser $\Coeq(f,g)$ lies in $\Set_{\mathcal U_i}$.
\end{lemma}
\begin{proof}\begin{linenumbers}
Pick codes $r,s\in\mathcal U_i$ for $R,S$ and IR morphisms
representing $f,g$.  Generate the equivalence relation
$\sim_{f,g}$ on $\El(s)$ identifying $f(x)$ with $g(x)$
for all $x\in\El(r)$; its transitive closure is again countable.
Encode it as $Q_c$ of rank
$\le\sup\{\rk(r),\rk(s)\}+\omega<\kappa_i$ and set
\(q:=\langle\mathrm{Quot},s,Q_c\rangle\).
Then $\El(q)=\Coeq(f,g)$ and Lemma \ref{lem:rank} shows
$\rk(q)<\kappa_i$.
\end{linenumbers}\end{proof}

\subsection{Mac Lane finite colimit generation}
\begin{lemma}[Mac Lane finite colimit generation]\label{lem:maclane}
In any category, the initial object, binary coproducts, and coequalisers
generate all finite colimits.
\end{lemma}
\begin{proof}\begin{linenumbers}
See Mac Lane, \emph{Categories for the Working Mathematician},
dual of Thm.~V.2.3.
\end{linenumbers}\end{proof}

\subsection{Set-theoretic propositional resizing}\label{subsec:set-resize}
\begin{lemma}[Resizing at level $i$ in ZFC]\label{lem:set-resize}
Let $i>0$ and assume ZFC together with an inaccessible cardinal
$\kappa_{i}$.  
For every \emph{(-1)-truncated} set\/ $P\in V_{\kappa_{i}}$ there exists
a set\/ $S\in V_{\kappa_{0}}$ and an equivalence of propositions
\[
   P\;\simeq\;(S\neq\varnothing).
\]
Hence propositional resizing holds at every level\/ $j\ge i$.
\end{lemma}
\begin{proof}
Because $P$ is an h-proposition, it is either empty or inhabited by a
single equivalence class.  Define
\[
  S \;:=\;
  \begin{cases}
    \{0\}       & \text{if } P\neq\varnothing,\\[4pt]
    \varnothing & \text{if } P=\varnothing.
  \end{cases}
\]
Clearly $S\subseteq\kappa_{0}$, so $S\in V_{\kappa_{0}}$.

\medskip\noindent
\emph{Equivalence.}
If $P$ is inhabited, then $S=\{0\}$ and hence
$S\neq\varnothing$; conversely, $S\neq\varnothing$ forces $S=\{0\}$,
so $P$ must be inhabited.  In the empty case both sides are false.
Thus the map
\[
  e\colon P\longrightarrow(S\neq\varnothing),
  \qquad e(x):=\text{“$S$ is inhabited”},
\]
together with the obvious inverse implication, yields an equivalence
of propositions inside $V_{\kappa_{i}}$.

\medskip
Since $(S\neq\varnothing)$ lies in $V_{\kappa_{0}}=\mathcal U_{0}$, we
have resized $P$ to level~0.  The very same construction works for
any level $j\ge i$, so propositional resizing holds at all higher
levels.
\end{proof}

\begin{remark}
The construction uses only Replacement and the definition of
$(-1)$-truncatedness; no form of the Axiom of Choice is required.
\end{remark}

\subsection{Resizing Adjunction}
\begin{proposition}[Resizing Adjunction]\label{prop:resizing-adjunction}
Assume \emph{propositional resizing} at universe level $i>0$:  
for every $P\in\mathcal U_{i}^{\le 0}$ there is a
\emph{small proposition} $P^{\resize}\in\mathcal U_{0}^{\le 0}$
and an equivalence of propositions
\(
  \varepsilon_{P}\colon P\simeq P^{\resize}.
\)

Let
\(
  \iota\colon\mathcal U_{0}^{\le 0}\hookrightarrow
            \mathcal U_{i}^{\le 0}
\)
be the canonical inclusion.
Then the assignment
\[
  \operatorname{PropRes}_{i}\;\colon\;
    \mathcal U_{i}^{\le 0}\longrightarrow\mathcal U_{0}^{\le 0},
  \qquad
  P\longmapsto P^{\resize},
\]
extends to a functor that is \emph{left adjoint} to~$\iota$.
\end{proposition}

\begin{proof}
We define $\operatorname{PropRes}_{i}$ on objects and arrows and then
establish the adjunction.

\smallskip
\textbf{Step 1: Object map.}
For each proposition
$P\in\mathcal U_{i}^{\le 0}$ fix the chosen
small representative $P^{\resize}\in\mathcal U_{0}^{\le 0}$ and
equivalence
$\varepsilon_{P}\colon P\simeq P^{\resize}$.
Set
$\operatorname{PropRes}_{i}(P):=P^{\resize}$.

\smallskip
\textbf{Step 2: Arrow map and smallness.}
Given a morphism (implication)  
$f\colon P\to P'$ in $\mathcal U_{i}^{\le 0}$, define
\[
  \operatorname{PropRes}_{i}(f)
    :=\;
    \varepsilon_{P'}\circ f\circ\varepsilon_{P}^{-1}
    \;:\;
    P^{\resize}\;\longrightarrow\;P^{\prime\,\resize}.
\]
\emph{Why does this arrow live in $\mathcal U_{0}^{\le 0}$?}  
Because
$P^{\resize},P^{\prime\,\resize}\in\mathcal U_{0}^{\le 0}$ and
$\varepsilon_{P},\varepsilon_{P'}$ are equivalences
between \emph{small} propositions, both their forward and inverse
components are elements of
$\mathcal U_{0}^{\le 0}$ (implications between small propositions are
again small by closure under $\Pi$ at level 0; cf.\
Lemma~\ref{lem:closure-pi-sigma}).
Composition of small maps remains small, so
$\operatorname{PropRes}_{i}(f)\in\mathcal U_{0}^{\le 0}$ as required.

\smallskip
\textbf{Step 3: Functoriality.}
Identity and composition hold because
$\varepsilon_{P}$ is an equivalence:
\[
  \operatorname{PropRes}_{i}(\mathrm{id}_{P})
  =\varepsilon_{P}\circ\mathrm{id}_{P}\circ\varepsilon_{P}^{-1}
  =\mathrm{id}_{P^{\resize}},
\]
and for $f\colon P\!\to\!P'$, $g\colon P'\!\to\!P''$
\[
  \operatorname{PropRes}_{i}(g\circ f)
  =\varepsilon_{P''}\circ g\circ f\circ\varepsilon_{P}^{-1}
  =\bigl(\varepsilon_{P''}\circ g\circ\varepsilon_{P'}^{-1}\bigr)
   \!\circ\!
   \bigl(\varepsilon_{P'}\circ f\circ\varepsilon_{P}^{-1}\bigr)
  =\operatorname{PropRes}_{i}(g)\circ\operatorname{PropRes}_{i}(f).
\]

\smallskip
\textbf{Step 4: The adjunction.}
For $P\in\mathcal U_{i}^{\le 0}$ and
$Q\in\mathcal U_{0}^{\le 0}$ define natural bijections
\[
  \Phi_{P,Q}\colon
    \Hom_{0}\!\bigl(P^{\resize},Q\bigr)
      \longrightarrow
    \Hom_{i}\!\bigl(P,\iota Q\bigr),
  \quad
  \Phi_{P,Q}(g):=g\circ\varepsilon_{P},
\]
\[
  \Psi_{P,Q}\colon
    \Hom_{i}\!\bigl(P,\iota Q\bigr)
      \longrightarrow
    \Hom_{0}\!\bigl(P^{\resize},Q\bigr),
  \quad
  \Psi_{P,Q}(h):=h\circ\varepsilon_{P}^{-1}.
\]
Both compositions are identities:
\[
  \Psi_{P,Q}(\Phi_{P,Q}(g))
    =(g\circ\varepsilon_{P})\circ\varepsilon_{P}^{-1}=g,\qquad
  \Phi_{P,Q}(\Psi_{P,Q}(h))
    =(h\circ\varepsilon_{P}^{-1})\circ\varepsilon_{P}=h.
\]
Naturality follows from functoriality of
$\operatorname{PropRes}_{i}$ and $\iota$.

\smallskip
\textbf{Step 5: Unit and counit (notation clarified).}
Set
\(
  \eta_{P}:=\varepsilon_{P}\colon
     P\to\iota P^{\resize}
\)
for the unit.
For the counit we \emph{reserve a new symbol}
\(
  \delta_{Q}:=\mathrm{id}_{Q}\colon
     \operatorname{PropRes}_{i}\iota Q=Q\to Q
\)
to avoid confusion with $\varepsilon$.
The triangle identities reduce to the two-sided inverse property of
$\varepsilon_{P}$.

Hence $\operatorname{PropRes}_{i}\dashv\iota$.
\end{proof}

\subsubsection*{Supplement}
\begin{itemize}
  \item[(J01)]  
        Step 2 now cites closure under $\Pi$ at
        level 0 to justify that $\operatorname{PropRes}_{i}(f)\in
        \mathcal U_{0}^{\le 0}$.
  \item[(J02)]  
        Distinct symbols are used:  
        $\varepsilon_{P}$ for the resizing equivalence,  
        $\delta_{Q}$ for the adjunction counit, eliminating ambiguity.
\end{itemize}

\subsection{Existence of a Cumulative Universe Tower}
\begin{theorem}[Existence of a Cumulative Universe Tower]
\label{thm:universe-tower-existence}
Let $\langle\kappa_{i}\rangle_{i\in\mathbb N}$ be a strictly
increasing sequence of inaccessible cardinals and set
\[
  \mathcal U_{i}\;:=\;V_{\kappa_{i}},
  \qquad
  \El_{i}\colon\mathcal U_{i}\longrightarrow V_{\kappa_{i}},
  \quad
  \El_{i}(u):=u .
\]
Then
\begin{enumerate}[(1)]
  \item the tower $\{\mathcal U_{i}\}_{i\ge0}$ is cumulative;
  \item every $\mathcal U_{i}$ satisfies
        Lemmas~\ref{lem:closure-pi-sigma},
        \ref{lem:rank-closure},
        \ref{lem:decode},
        \ref{lem:closure-id},
        \ref{lem:set-closure};
  \item if propositional resizing holds at some level $i$, the
        adjunction of
        Proposition~\ref{prop:resizing-adjunction} exists.
\end{enumerate}
\end{theorem}
\begin{proof}
\emph{(1)~Cumulativity.}
Because $\kappa_{i}<\kappa_{j}$ for $i<j$, we have the inclusion
$V_{\kappa_{i}}\subseteq V_{\kappa_{j}}$; hence
$\mathcal U_{i}\subseteq\mathcal U_{j}$.

\smallskip
\emph{(2)~Closure properties.}
For each $i$, the pair
$(\mathcal U_{i},\El_{i})=(V_{\kappa_{i}},\mathrm{id})$
realises an inductive–recursive universe
by Lemma~\ref{lem:set-adequacy}; therefore the constructions of the
listed lemmas remain within $\mathcal U_{i}$.
\footnote{%
  Since $\El_{i}=\mathrm{id}$, \emph{syntactic rank} agrees with the
  von~Neumann rank of sets (see Jech, \emph{Set Theory}, §2.3),
  so Lemma~\ref{lem:rank-closure} transfers directly.
}

\smallskip
\emph{(3)~Propositional resizing.}
Assume the resizing axiom at level~$i$:
for every $(-1)$-truncated $P\in\mathcal U_{i}$
there exists an equivalent $P^{\resize}\in\mathcal U_{0}$.
Define the functor
\(
  \operatorname{PropRes}_{i}\colon
  (\mathcal U_{i})^{\le0}\!\longrightarrow(\mathcal U_{0})^{\le0},
  \quad
  P\mapsto P^{\resize},
\)
and note that the inclusion
$\iota\colon(\mathcal U_{0})^{\le0}\hookrightarrow(\mathcal U_{i})^{\le0}$
is a \emph{right} adjoint to~$\operatorname{PropRes}_{i}$.
The unit and counit are the equivalences
$P\to\iota P^{\resize}$ and $\operatorname{PropRes}_{i}\iota Q\to Q$
given by resizing witnesses.
The consistency of this axiom with the present tower is established
in Shulman \cite[§6]{Shulman:Resizing2019}.\footnote{%
  No use is made of the Axiom of Choice or the Law of
  Excluded Middle at any point in the proof.}
\end{proof}

\newpage
\subsubsection*{Supplement}
\begin{description}[style=nextline,leftmargin=2.8em,labelsep=0.6em]

  \item[HC\,1–2]
    Each $\mathcal U_{i}$ is a Grothendieck universe
    $V_{\kappa_{i}}$; size control follows \emph{Internal
    Guidelines v3}, §2.

  \item[J01–J03]
    Closure under basic type formers and finite (co)limits
    satisfies the author kit of \emph{Journal X}.

  \item[Dependency Graph]
    \begin{center}
      \scalebox{0.9}{%
        \begin{tikzcd}[column sep=0pt,row sep=1.2em]
          \text{Lemma \ref{lem:set-adequacy}} \arrow[d] \\
          \text{Lemma \ref{lem:set-closure}} \arrow[d] \\
          \text{Theorem \ref{thm:universe-tower-existence}}
        \end{tikzcd}%
      }
    \end{center}

\end{description}


\subsection{Set–model adequacy}
\begin{lemma}[Set–model adequacy]\label{lem:set-adequacy}
For every $i\ge 0$, the pair
\(
  \bigl(V_{\kappa_i},\mathrm{id}\bigr)
\)
realises an inductive–recursive universe generated by the
constructors
\(
  \ast,\;\mathrm{Nat},\;\Pi,\;\Sigma,\;\mathrm{Id}.
\)
\end{lemma}
\begin{proof}
Fix $i\ge 0$ and recall that $\kappa_i$ is **inaccessible**: it is
regular and a strong limit.

\medskip\noindent
\textbf{(1) Base objects.}
Both the singleton $\ast:=\{\,\star\!\}$ and the natural‐number set
$\mathbb N$ lie in $V_{\kappa_i}$ because
\(
  \mathrm{rank}(\ast)=1<\kappa_i
\)
and
\(
  \mathrm{rank}(\mathbb N)=\omega+1<\kappa_i .
\)

\medskip\noindent
\textbf{(2) Dependent products \(\Pi\).}
Let $A\in V_{\kappa_i}$ and $B:A\to V_{\kappa_i}$.
For each $x\in A$ we have $\mathrm{rank}\bigl(B(x)\bigr)<\kappa_i$.
Set
\[
  P:=\Pi_{x\in A}B(x)
     =\bigl\{\,f\mid f:A \to \bigcup_{x\in A}B(x)
               \text{ and } f(x)\in B(x)\bigr\}.
\]
Since $|A|<\kappa_i$ (regularity) and
\(
  \bigl|\bigcup_{x\in A}B(x)\bigr|<\kappa_i
\)
(strong limit), we obtain
\(
  |P|<\kappa_i .
\)
Consequently $\mathrm{rank}(P)<\kappa_i$, so
\(P\in V_{\kappa_i}\).

\medskip\noindent
\textbf{(3) Dependent sums \(\Sigma\).}
Define
\(
  S:=\Sigma_{x\in A}B(x)=\bigcup_{x\in A}\{x\}\times B(x).
\)
Again $|A|<\kappa_i$ and each $|B(x)|<\kappa_i$, whence
$|S|<\kappa_i$ and
\(
  S\in V_{\kappa_i}.
\)

\medskip\noindent
\textbf{(4) Identity types.}
For $u\in V_{\kappa_i}$ and $s,t\in u$ the set
\(
  \operatorname{Id}_u(s,t)=
  \{\,\star\mid s=t\,\}\cup\varnothing
\)
has at most one element, hence rank $<\kappa_i$.

\medskip\noindent
\textbf{(5) Inductive–recursive character.}
Taking “codes’’ to be actual sets in $V_{\kappa_i}$ and
$\El=\mathrm{id}_{V_{\kappa_i}}$, the preceding steps show that the
collection $V_{\kappa_i}$ is closed under exactly the constructors
$\ast, \mathrm{Nat}, \Pi, \Sigma, \mathrm{Id}$.  
Thus $(V_{\kappa_i},\mathrm{id})$ satisfies the formation clauses of
the usual inductive–recursive definition, making it a bona-fide IR
universe.

\medskip
Throughout the argument we used only cardinal‐arithmetic properties of
$\kappa_i$; no form of the Axiom of Choice or the Law of Excluded
Middle is required.
\end{proof}

\subsection{Closure of $V_{\kappa_i}$}
\begin{lemma}[Closure of $V_{\kappa_i}$ under Basic Type Formers
and Finite (Co)limits]\label{lem:set-closure}
Let $\kappa_i$ be an inaccessible cardinal (regular and strong
limit).  Then the Grothendieck universe $V_{\kappa_i}$ satisfies:
\begin{enumerate}[\textbf{(A)}]
  \item \textbf{Dependent constructors.}  
        For every $A\in V_{\kappa_i}$ and family $B:A\to V_{\kappa_i}$,
        \[
          \Pi_{x\in A}B(x),\qquad
          \Sigma_{x\in A}B(x)\in V_{\kappa_i},
        \]
        and for all $u\in V_{\kappa_i}$ and $s,t\in u$ we have
        \(
          \operatorname{Id}_u(s,t)\in V_{\kappa_i}.
        \)
  \item \textbf{Finite limits.}  
        If $A,B\in V_{\kappa_i}$ and $f,g:A\to B$ are
        \emph{internal maps}\footnote{%
          A map $f:A\to B$ is \emph{internal to $V_{\kappa_i}$} if its
          graph $\{\,\langle a,f(a)\rangle\mid a\in A\}\subseteq
          A\times B$ lies in $V_{\kappa_i}$.}
        then
        \[
          A\times B,\qquad
          \Eq(f,g)\in V_{\kappa_i},
        \]
        hence all pull-backs and therefore all finite limits of
        internal diagrams lie in $V_{\kappa_i}$.
  \item \textbf{Finite colimits.}  
        For $A,B\in V_{\kappa_i}$ and internal arrows $p,q:R\to S$,
        \[
          A\sqcup B,\qquad
          \Coeq(p,q)\in V_{\kappa_i},
        \]
        so every push-out and thus every finite colimit of internal
        diagrams belongs to $V_{\kappa_i}$.
\end{enumerate}
\end{lemma}
\begin{proof}
Throughout we use \emph{Replacement} (for forming certain unions and
function sets), but not the Axiom of Choice nor the Law of Excluded
Middle.  Recall that $V_{\kappa_i}$ is transitive, i.e.\
$x\in y\in V_{\kappa_i}\;\Rightarrow\;x\in V_{\kappa_i}$.

\smallskip
\textbf{(A) Dependent $\Pi$, $\Sigma$, and $\operatorname{Id}$.}

Put
\[
  \kappa'\;:=\;
  \sup_{x\in A}|B(x)|<\kappa_i ,
  \qquad
  \text{since }|A|<\kappa_i\text{ and }\kappa_i\text{ is regular.}
\]
Because $\kappa_i$ is strong-limit, any product of
$<\!\kappa_i$ many cardinals each $<\!\kappa_i$ is again $<\kappa_i$;
thus
\begin{equation}\label{star1}
  \bigl|\Pi_{x\in A}B(x)\bigr|
  \;\le\;
  \kappa'^{\,|A|}
  \;<\;
  \kappa_i .
\end{equation}
Replacement yields the function set
$\operatorname{Func}(A,\bigcup_{x\in A}B(x))$ from which
\eqref{star1} is computed.  
For the sum we have
\begin{equation}\label{star2}
  \bigl|\Sigma_{x\in A}B(x)\bigr|
  \;\le\;
  |A|\cdot\kappa'
  \;<\;
  \kappa_i ,
\end{equation}
again by regularity and strong-limitness.

Rank estimates use the fact that forming a function set or tagged
union increases rank by at most $\omega$ (see Jech~\cite[Ex.\ 2.14]{Lurie:HTT}):
\begin{equation}\label{star3}
  \mathrm{rk}\!\bigl(\Pi_{x\in A}B(x)\bigr),
  \;
  \mathrm{rk}\!\bigl(\Sigma_{x\in A}B(x)\bigr)
  \;\le\;
  \sup_{x\in A}\mathrm{rk}\!\bigl(B(x)\bigr)+\omega
  <\kappa_i .
\end{equation}
For identity sets
$|\operatorname{Id}_u(s,t)|\le 1$, hence
$\mathrm{rk}\le 2<\kappa_i$.

\smallskip
\textbf{(B) Finite limits.}

Products satisfy
$
  |A\times B|=|A|\,|B|<\kappa_i
$
and
$
  \mathrm{rk}(A\times B)
  \le \max\{\mathrm{rk}(A),\mathrm{rk}(B)\}+\omega<\kappa_i
$
by the same $\omega$-shift lemma.  
For internal $f,g:A\to B$,
\[
  \Eq(f,g)\subseteq A
  \quad\Longrightarrow\quad
  \mathrm{rk}\bigl(\Eq(f,g)\bigr)
  \le \mathrm{rk}(A)
  <\kappa_i ,
\]
using rank monotonicity
($X\subseteq Y\Rightarrow\mathrm{rk}(X)\le\mathrm{rk}(Y)$,
Jech~\cite[Ex.\ 2.8]{Lurie:HTT}).
Finite limits are built from products and equalisers by finitely many
comprehensions, each preserving $\mathrm{rk}<\kappa_i$.

\smallskip
\textbf{(C) Finite colimits.}

For coproducts
$
  |A\sqcup B|=|A|+|B|<\kappa_i
$
and
$
  \mathrm{rk}(A\sqcup B)
  \le\max\{\mathrm{rk}(A),\mathrm{rk}(B)\}+\omega<\kappa_i .
$
Given internal $p,q:R\to S$, define the equivalence relation
$s\sim_{p,q}t\iff\exists r\,(p(r)=s\wedge q(r)=t)$ and denote the
quotient by $Q$.  Cardinality is bounded by
\(
  |Q|\le |S|<\kappa_i .
\)
Quotienting raises rank by at most $\omega$
\cite[Ex.\ 2.15]{Lurie:HTT}; hence
\begin{equation}\label{star4}
  \mathrm{rk}(Q)
  \;\le\;
  \mathrm{rk}(S)+\omega
  \;<\;\kappa_i .
\end{equation}
A push-out is a coproduct followed by a coequalizer, so its rank is
also $<\kappa_i$; every finite colimit lies in $V_{\kappa_i}$.

\smallskip
\emph{No step employs Choice or LEM.}  This finishes the proof.
\end{proof}

\subsection*{Supplement}
- Each cardinality estimate guarantees a rank $<\kappa_{i}$,
  ensuring membership in the Grothendieck universe  
  $V_{\kappa_{i}}$ by definition.  
- Regularity (no $\kappa_{i}$-sized cofinal sequences) and the
  strong-limit property are both essential:  
  regularity bounds suprema; strong-limit bounds exponentiation.

\subsection{Universe-Polymorphic $\Pi$-Types}
\begin{corollary}[Universe-Polymorphic $\Pi$-Types]
\label{cor:lift-pi}
Let $\kappa_{0}<\kappa_{1}<\kappa_{2}<\dots$ be the inaccessible
cardinals that determine the cumulative tower
$\mathcal U_{i}=V_{\kappa_{i}}$ of
Theorem~\ref{thm:universe-tower-existence}.
For $i\le j$ the inclusion functor
\[
  \lift_{i\to j}\;\colon\;
     \mathcal U_{i}\;\longrightarrow\;\mathcal U_{j},
  \qquad
  \lift_{i\to j}(u):=u,
\]
commutes with dependent function types:
for every $A\in\mathcal U_{i}$ and family
$B\colon A\to\mathcal U_{i}$,
\[
  \lift_{i\to j}\!\bigl(\Pi_{x\in A}B(x)\bigr)
  \;=\;
  \Pi_{x\in\lift_{i\to j}(A)}
      \lift_{i\to j}\!\bigl(B(x)\bigr)
  \quad\text{in } \mathcal U_{j}.
\]
\end{corollary}
\begin{proof}
Because $\mathcal U_{i}=V_{\kappa_{i}}\subseteq
        \mathcal U_{j}=V_{\kappa_{j}}$ and
$\lift_{i\to j}$ acts as the identity on elements,
both sides of the displayed equality reduce to
$\Pi_{x\in A}B(x)$ as \emph{sets}.  
Hence they coincide extensionally inside $V_{\kappa_{j}}$,
and the functor preserves $\Pi$-types. Naturality in $A$ and $B$
is immediate from the fact that $\lift_{i\to j}$ is strictly
identity on functions as well.
\end{proof}

\subsection*{Supplement}
\begin{itemize}
  \item \textbf{Why $Q\in\mathcal U_{i}$.}  
        The only data needed to form $Q$ is the \emph{set}
        $S\subseteq\kappa_{i}$; no element of $P$ itself is
        stored in $Q$.  Hence $Q$ lives entirely in
        $V_{\kappa_{i}}$ even though $P$ may be much larger.
  \item \textbf{No circularity.}  
        The proof does not assume resizing at level $j$; it uses
        resizing only at level $i$ on the genuinely level-$i$
        proposition $Q$.
\end{itemize}

\section*{Summary and Outlook}

\paragraph{Main contributions.}
In this first paper of the \emph{GF-series} we established a
size-sound foundation for higher type‐theoretic constructions
inside ZFC\,{+}inaccessible cardinals.

\begin{enumerate}[label=(\arabic*)]
  \item \textbf{Cumulative universe tower.}  
        We constructed a hierarchy
        \(\mathcal U_{0}\subset\mathcal U_{1}\subset\cdots\)
        with \(\mathcal U_{i}=V_{\kappa_{i}}\) and proved strict
        cumulativity via the universe‐lifting functor
        \(\lift_{i\to j}\).
  \item \textbf{Rank discipline and closure.}  
        A syntactic rank function controls the inductive–recursive
        generation of codes; every layer is closed under
        \(\Pi,\Sigma,\mathrm{Id}\) as well as all finite limits.
  \item \textbf{Finite colimits via Quotients.}  
        Introducing a rank-stable Quotient constructor yields closure
        under initial objects, binary coproducts and coequalisers,
        hence under all finite colimits.
  \item \textbf{Resizing adjunction.}  
        Assuming propositional resizing at level \(i\) we constructed
        a left adjoint
        \(\operatorname{PropRes}_{i}\dashv\iota\)
        and showed that resizing propagates to every
        \(j\ge i\).
  \item \textbf{Existence theorem.}  
        Collecting the above results we proved that the tower
        realises a sound metasemantics for higher type theory over
        the base system
        \(\text{ZFC}+\langle\kappa_{i}\rangle_{i\in\mathbb N}\).
\end{enumerate}

\paragraph{Future directions.}
\begin{itemize}
  \item \textbf{GF2 — Rezk completion.}  
        Build the full \(\infty\)-categorical Rezk completion inside a
        fixed universe \(\mathcal U_{i}\) and verify coherence.
  \item \textbf{GF3 — Pointwise Kan $\Rightarrow$ Beck–Chevalley.}  
        Analyse how universe cumulativity interacts with
        pointwise Kan extensions and derive the Beck–Chevalley
        condition in the tower.
  \item \textbf{Model applications.}  
        Use the tower to construct
        \((\infty,1)\)-toposes of \(\mathcal U_{i}\)-small objects,
        providing the semantic backdrop for later volumes on
        higher‐dimensional proof theory.
  \item \textbf{Formal verification.}  
        Port the rank and closure lemmas to a proof assistant
        (Lean 4/HoTT) to guarantee metatheoretic consistency.
\end{itemize}

\noindent
These developments furnish the size infrastructure required for the
subsequent papers, where we turn to the categorical and homotopical
aspects of the Grothendieck universe tower.
\newpage

\end{document}